\documentclass[a4paper,english]{extarticle}
\usepackage[latin9]{inputenc}
\usepackage{babel}
\usepackage{amsthm}
\usepackage{amsmath}
\usepackage{amssymb}
\usepackage{esint}
\usepackage[unicode=true,pdfusetitle,
 bookmarks=true,bookmarksnumbered=false,bookmarksopen=false,
 breaklinks=false,pdfborder={0 0 1},backref=false,colorlinks=false]
 {hyperref}

\makeatletter

\pdfpageheight\paperheight
\pdfpagewidth\paperwidth

\numberwithin{equation}{section}
\numberwithin{figure}{section}
\theoremstyle{plain}
\newtheorem{thm}{\protect\theoremname}[section]
\theoremstyle{remark}
\newtheorem{rem}[thm]{\protect\remarkname}
\theoremstyle{plain}
\newtheorem{lem}[thm]{\protect\lemmaname}
\ifx\proof\undefined
\newenvironment{proof}[1][\protect\proofname]{\par
\normalfont\topsep6\p@\@plus6\p@\relax
\trivlist
\itemindent\parindent
\item[\hskip\labelsep\scshape #1]\ignorespaces
}{%
\endtrivlist\@endpefalse
}
\providecommand{\proofname}{Proof}
\fi

\usepackage{cancel}
\numberwithin{equation}{section}
\date{}

\makeatother

\providecommand{\lemmaname}{Lemma}
\providecommand{\remarkname}{Remark}
\providecommand{\theoremname}{Theorem}

\begin{document}

\title{Boundary determination of the Lamé moduli for the isotropic elasticity
system}

\author{Yi-Hsuan Lin\thanks{{\footnotesize{}{}{}Department of Mathematics, University of Washington,
Seattle, USA}}; Gen Nakamura\thanks{{\footnotesize{}{}{}Hokkaido University, Sapporo 060-0808, Japan}}}
\maketitle
\begin{abstract}
We consider the inverse boundary value problem of determining the
Lamé moduli of an isotropic, static elasticity equations of system
at the boundary from the localized Dirichlet-to-Neumann map. Assuming
appropriate local regularity assumptions as weak as possible on the
Lamé moduli and on the boundary, we give explicit pointwise reconstruction
formulae of the Lamé moduli and their higher order derivatives at
the boundary from the localized Dirichlet-to-Neumann map. 
\end{abstract}
\textbf{Key words}: Inverse boundary value problem, Dirichlet-to-Neumann
map, isotropic elasticity system, boundary determination, Stroh formalism\\
\textbf{Mathematics Subject Classification}: 74B05, 35R30

\section{Introduction and main result}

Let us briefly give our main result before giving its detailed mathematical
description. That is we give explicit pointwise reconstruction formulae
of Lamé moduli and their derivatives at a given point on the boundary
from the measured data called the localized Dirichlet-to-Neumann map
for the inverse boundary value problem associated to an isotropic
elastic equation in a bounded domain. We will refer this kind of inverse
problem by \textit{boundary determination}.

Let $\Omega\subset\mathbb{R}^{3}$ be a bounded domain with boundary
$\partial\Omega$ and $\lambda=\lambda(x),\mu=\mu(x)$ be the Lamé
moduli which satisfy 
\begin{equation}
\mu>0,\mbox{ }3\lambda+2\mu>0\,\,\mbox{ on}\,\,\overline{\Omega}.\label{eq:Strong Convex}
\end{equation}
The regularity of $\partial\Omega$ and Lamé moduli will be specified
later. Consider the boundary value problem 
\begin{equation}
\begin{cases}
(\mathcal{L}u)_{i}:=\sum_{j,k,l=1}^{3}\dfrac{\partial}{\partial x_{j}}(\dot{C}_{ijkl}\dfrac{\partial}{\partial x_{l}}u_{k})=0\mbox{ }(i=1,2,3) & \mbox{ in }\Omega,\\
u=f\in H^{1/2}(\partial\Omega;\mathbb{C}^{3}) & \mbox{ on }\partial\Omega
\end{cases}\label{eq:Elasticity}
\end{equation}
for the displacement vector $u=(u_{1},u_{2},u_{3})$, where 
\begin{equation}
\dot{C}_{ijkl}=\dot{C}_{ijkl}(x)=\lambda\delta_{ij}\delta_{kl}+\mu(\delta_{ik}\delta_{jl}+\delta_{il}\delta_{jk})\mbox{ }(1\leq i,j,k,\ell\leq3)\label{eq:isotrpic elastic tensor}
\end{equation}
are isotropic elastic tensors in terms of the Cartesian coordinates
$x=(x_{1},x_{2},x_{3})$ with Kronecker delta $\delta_{ij}$. It is
easy to see that $\dot{C}_{ijkl}$ defined by (\ref{eq:isotrpic elastic tensor})
satisfies the symmetry given as 
\[
\dot{C}_{ijkl}(x)=\dot{C}_{klij}(x)=\dot{C}_{jikl}(x)
\]
and the strong convexity condition given as 
\[
\sum_{i,j,k,l=1}^{3}\dot{C}_{ijkl}(x)\varepsilon_{ij}\varepsilon_{kl}\geq c_{0}\sum_{i,j=1}^{3}\varepsilon_{ij}^{2}
\]
with some constant $c_{0}>0$ for any $x\in\overline{\Omega}$ and
symmetric matrix $(\varepsilon_{ij})$.

Define the Dirichlet-to-Neumann (DN) map $\Lambda_{\mathcal{C}}:H^{1/2}(\partial\Omega)\to H^{-1/2}(\partial\Omega)$
by 
\[
\left(\Lambda_{\mathcal{C}}f\right)_{i}:=\sum_{j,k,l=1}^{3}\nu_{j}\dot{C}_{ijkl}\dfrac{\partial u_{k}}{\partial x_{l}}|_{\partial\Omega}\,\,\mbox{ for }\,\,i=1,2,3,
\]
where $u$ is the solution of (\ref{eq:Elasticity}) and $\nu=(\nu_{1},\nu_{2},\nu_{3})$
is the unit normal of $\partial\Omega$ directed into the exterior
of $\Omega$. Let $\epsilon(u):=(\epsilon_{ij}(u))$ be the strain
tensor associated to the solution $u$ of (\ref{eq:Elasticity}),
where 
\[
\epsilon_{ij}(u)=\dfrac{1}{2}\left(\dfrac{\partial u_{i}}{\partial x_{j}}+\dfrac{\partial u_{j}}{\partial x_{i}}\right)\mbox{ for }i,j=1,2,3.
\]
It is well-known that for $f,g\in H^{1/2}(\partial\Omega)$, 
\[
\left\langle \Lambda_{\mathcal{C}}f,g\right\rangle _{H^{-1/2}(\partial\Omega)\times H^{1/2}(\partial\Omega)}=\int_{\Omega}\left(\lambda\,\mbox{div}\,u\,(\overline{\mbox{div}\,v})+2\mu\,\epsilon(u):\overline{\epsilon(v)}\right)dy,
\]
where $\left\langle \cdot,\cdot\right\rangle _{H^{-1/2}(\partial\Omega)\times H^{1/2}(\partial\Omega)}$
is the pairing in $H^{-1/2}(\partial\Omega)\times H^{1/2}(\partial\Omega)$,
$v\in H^{1}(\Omega)$ can be taken whichever satisfies $v=g$ on $\partial\Omega$
and the notation ``$:$'' denotes the Frobenius inner product. Let
$x_{0}\in\partial\Omega$ be an arbitrary point, then the DN map $\Lambda_{\mathcal{C}}$
can be localized near $x_{0}$ by restricting the support of $f,g$
in an open neighborhood of $x_{0}$ in $\partial\Omega$.

The precise description of the aim of this paper is to recover $\lambda,\mu$
and their higher-order normal derivatives near a given point $x_{0}\in\partial\Omega$
by knowing the localized $\Lambda_{\mathcal{C}}$ around $x_{0}$,
under the regularity assumptions on $\lambda,\mu$ and $\partial\Omega$
near $x_{0}$ weak as possible. More specifically, we will show that
for any $m\in\mathbb{N}$, the Lamé moduli and their normal derivatives
at $x_{0}\in\partial\Omega$ up to order $m$ can be given explicitly
from the localized DN map $\Lambda_{\mathcal{C}}$ around $x_{0}$.
There are some related results on this boundary determination for
the elasticity system. For the two dimensional isotropic elastic system,
the boundary determination was given by \cite{akamatsu1991identification}
if the Lamé moduli and $\partial\Omega$ are smooth. In \cite{nakamura1999layer}
and \cite{nakamura1995inverse} the authors developed layer stripping
algorithm in which they solved the boundary determination for the
three dimensional isotropic and transversally isotropic elastic systems
also for the case the elasticity tensor and $\partial\Omega$ are
smooth. But it should be remarked here that a result of boundary determination
under regularity assumptions as weak as possible was missing for the
elasticity systems and we aimed to provide such a result for the isotropic
elasticity system in this paper. For practical application, it is
needless to say the importance of such a result. We note that there
are related results for both the isotropic and anisotropic conductivity
equations using arguments similar as in this paper. For that see \cite{brown2001recovering,kang2002boundary,nakamura2001direct,nakamura2001local}
and the references there in.

For any $m\in\mathbb{N}$, $p\in(0,1)$, $C^{m,p}(\Omega)$ denotes
the standard H$\ddot{\text{o}}$lder space. Then the regularity assumptions
on the Lamé moduli $\lambda,\,\mu$ and $\partial\Omega$ are locally
$C^{m+2}$ and $C^{m,p}$ near $x=0$, respectively.

By introducing the boundary normal coordinates which was used in \cite{kang2002boundary,lee1989determining,nakamura2001local}
for the conductivity equation and \cite{nakamura2017reconstruction}
for the elasticity equation, we can flatten $\partial\Omega$. Also,
without loss of generality, we may assume that $x_{0}$ can be the
origin. In terms of the boundary normal coordinates the displacement
vector $u=(u_{1},u_{2},u_{3})$ and isotropic elastic tensor $(\dot{C}_{ijkl})$
will undergo tensorial change which complicates the notations and
description of arguments. Hence, in order not to distract reader's
attention, we first focus on the reconstruction formulae for the Lamé
parameters for the flat boundary case at $0\in\partial\Omega$. We
will illustrate the non flat boundary case in the last section of
this paper. It will be shown there that the difference we will have
for the non flat boundary case is just coming from the change of coordinates
and normal vector.

To begin with assume that $\partial\Omega$ is flat near $0\in\partial\Omega$
and $\Omega$ is locally given as $\{y_{3}>0\}$ in terms of the Cartesian
coordinates $(y_{1},y_{2},y_{3})$. For the local determination of
the Lamé parameters at $0\in\partial\Omega$, we only need to assume
each $\dot{C}_{ijkl}$ is of $C^{m,p}$ class around the origin. Fix
$x=(x_{1},x_{2},0)\in\partial\Omega$ and define $\mathcal{C}^{m,x}:=(\dot{C}_{ijkl}^{m,x})$
by 
\begin{equation}
\dot{C}_{ijkl}^{m,x}(y):=\sum_{b<m}\dfrac{\partial_{y_{3}}^{b}\dot{C}_{ijkl}(y',x_{3})}{b!}y_{3}^{b}\mbox{ for }y\mbox{ near }x.\label{eq:C_bx}
\end{equation}
Then extending this $\mathcal{C}^{m,x}$ to $\overline{\Omega}$ without
destroying the regularity and strong convexity, we denote the corresponding
localized DN map by $\Lambda_{\mathcal{C}^{m,x}}$. Similarly, we
define $\lambda^{m,x}$ and $\mu^{m,x}$ by 
\begin{equation}
\lambda^{m,x}(y):=\sum_{b<m}\dfrac{\partial_{y_{3}}^{b}\lambda(y',x_{3})}{b!}y_{3}^{b},\mbox{ }\mu^{m,x}(y):=\sum_{b<m}\dfrac{\partial_{y_{3}}^{b}\mu(y',x_{3})}{b!}y_{3}^{b}.\label{eq:Lame B}
\end{equation}
Let $\omega'=(\omega_{1},\omega_{2},0)$ be a unit tangent vector
of $\partial\Omega$ at $0$ and $\eta(y')\in C_{0}^{\infty}(\mathbb{R}^{2})$
satisfy 
\[
0\leq\eta\leq1,\mbox{ }\int_{\mathbb{R}^{2}}\eta^{2}dy'=1\mbox{ and supp}\eta\subset\{|y'|<1\}.
\]
First, by choosing suitably large $\ell\in\mathbb{N}$, we may assume
that $\dfrac{m}{\ell}=\dfrac{1}{\widetilde{\rho}}$ for some large
$\widetilde{\rho}\in\mathbb{N}$ with $\dfrac{1}{\widetilde{\rho}}<p$
and we also assume that 
\begin{equation}
(1-\dfrac{1}{\widetilde{\rho}})(m+p)\geq m+\dfrac{1}{\widetilde{\rho}}.\label{eq:smallness condition}
\end{equation}
For convenience, we denote $\rho=\dfrac{1}{\widetilde{\rho}}$ and
for large $N\in\mathbb{N}$, we put $\eta^{N}(y'):=\eta(N^{1-\rho}y')$.
For any column vector $\mathbf{a}=(a_{1},a_{2},a_{3})\in\mathbb{C}^{3}$,
let 
\begin{equation}
\phi^{N}(y):=\eta^{N}(y')\exp(\sqrt{-1}Ny'\cdot\omega')\mathbf{a}\label{eq:boundary test function}
\end{equation}
be the localized Dirichlet data around $0\in\partial\Omega$, then
we have the following theorem.
\begin{thm}
\label{Main Thm}(1) Let $\Omega$ be of $C^{1}$ class near $0\in\partial\Omega$
and let $\dot{C}_{ijkl}$ be continuous near $y=0$. Then 
\begin{equation}
\lim_{N\to\infty}\left\langle \Lambda_{\mathcal{C}}\phi^{N},\overline{\phi^{N}}\right\rangle =\sum_{i,j=1}^{3}Z_{ij}(0)a_{i}\overline{a_{j}}\label{eq:Theorem_part 1}
\end{equation}
and $Z_{ij}=\overline{Z_{ji}}$ for $1\leq i,j\leq3$, where 
\begin{eqnarray}
Z_{ii} & = & \dfrac{\mu}{\lambda+3\mu}\big(2(\lambda+2\mu)-(\lambda+\mu)\iota_{i}\big),\nonumber \\
Z_{ij} & = & \dfrac{\mu}{\lambda+3\mu}\big(-(\lambda+\mu)\iota_{i}\iota_{j}+\sqrt{-1}(-1)^{k}2\mu\,\iota_{k}\big),\mbox{ }1\leq i<j\leq3\label{eq:Z_ij}
\end{eqnarray}
with $(\iota_{1},\iota_{2},\iota_{3})=(\omega_{2},-\omega_{1},0)$
and the index $k\in\mathbb{N}$ has to satisfy $1\leq k\leq3$, $k\neq i,j$.

(2) For $m\in\mathbb{N}$, let $\partial\Omega$ be of $C^{m+2}$
class near $0\in\partial\Omega$. Let $\mathcal{C}=(\dot{C}_{ijkl})$
be of $C^{m,p}$ near $0$. Then 
\begin{align}
 & \lim_{N\to\infty}N^{m}\left\langle (\Lambda_{\mathcal{C}}-\Lambda_{\mathcal{C}^{m,0}})\phi^{N},\overline{\phi^{N}}\right\rangle \label{eq:Theorem_part 2}\\
= & \dfrac{1}{2^{m+1}}\dfrac{\partial^{m}\lambda}{\partial y_{3}^{m}}(0)\left(\sqrt{-1}\sum_{i=1}^{2}\omega_{i}a_{i}-a_{3}\right)^{2}\nonumber \\
 & +\dfrac{1}{2^{m}}\dfrac{\partial^{m}\mu}{\partial y_{3}^{m}}(0)\left[\sum_{i,j=1}^{2}\left(\dfrac{a_{i}\omega_{j}+a_{j}\omega_{i}}{2}\right)^{2}+2\sum_{i=1}^{2}\left(\dfrac{\sqrt{-1}a_{3}\omega_{i}-a_{i}}{2}\right)^{2}+1\right].\nonumber 
\end{align}
Hence from these formulae, we can recover Lamé moduli and their derivatives
up to order $m$.\end{thm}
\begin{rem}
We remark here that the above boundary determination formulae (\ref{eq:Theorem_part 1})
and (\ref{eq:Theorem_part 2}) are given in terms of the leading part
of the equations of system. Further (\ref{eq:Theorem_part 1}) was
proved in the Section 2 of \cite{tanuma2007stroh} and (\ref{eq:Z_ij})
was shown in Theorem 1.24 of \cite{tanuma2007stroh}, so we omit their
proofs. We will only prove (\ref{eq:Theorem_part 2}). 
\end{rem}
The rest of this paper is organized as follows. In Section \ref{Section 2},
we construct an approximate solution of $\mathcal{L}u=0$ with $u=\phi^{N}$
on $\partial\Omega$. By using this special solution, we will prove
Theorem 1.1 for the flat boundary case in Section \ref{Section 3}.
Finally, in Section \ref{Section 4}, we will demonstrate deriving
the reconstruction formulae in terms of the boundary normal coordinates
for the non-flat boundary case in the lasts section. 

\textbf{}\\
\textbf{Acknowledgments.} Y. H. Lin is partially supported by MOST
of Taiwan 160-2917-I-564-048. G. Nakamura would like to thank the
support from National Center for Theoretical Sciences (NCTS) for his
stay in National Taiwan University, Taipei, Taiwan and Grant-in-Aid
for Scientific Research (15K21766 and 15H05740) of the Japan Society
for the Promotion of Science.

\section{Construction of approximate solutions\label{Section 2}}

In order to prove Theorem 1, we need to construct an approximate solution
depending on a large parameter $N$. Let 
\begin{equation}
\Omega_{N}:=\left\{ y:|y_{1}|,|y_{2}|\leq N^{\rho-1},\mbox{ }0\leq y_{3}\leq\dfrac{1}{\sqrt{N}}\right\} .\label{eq:Omega N}
\end{equation}
and $\alpha$ be a multi-index such that 
\[
\alpha=(1-\rho,1-\rho,1)\mbox{ and }N^{\alpha}y=(N^{1-\rho}y_{1},N^{1-\rho}y_{2},Ny_{3})=(N^{1-\rho}y',Ny_{3}).
\]
Inspired by \cite{kang2002boundary,nakamura2001local}, we can prove
the following lemma. 
\begin{lem}
For each $N\in\mathbb{N}$, there exists for any column vector $\mathbf{a}=(a_{1},a_{2},a_{3})\in\mathbb{C}^{3}$
an approximate solution $\Phi^{N}$ of \eqref{eq:Elasticity} near
$0$ of the form 
\begin{equation}
\Phi^{N}(y)=e^{\sqrt{-1}Ny'\cdot\omega'}e^{-Ny_{3}}\left\{ \eta^{N}(y')\,\mathbf{a}+\sum_{n=1}^{\frac{m}{\rho}}N^{-n\rho}v_{n}(N^{\alpha}y)\right\} \label{eq:Approximate Solution}
\end{equation}
with $\Phi^{N}|_{\partial\Omega}=\phi^{N}=e^{\sqrt{-1}Ny'\cdot\omega'}\eta^{N}(y')\,\mathbf{a}$,
where each vector $v_{n}(N^{\alpha}y)$ is polynomial in $Ny_{3}$
with coefficients which are $C^{\infty}$-smooth functions of $N^{1-\rho}y'$
supported in $\{|y'|<N^{\rho-1}\}$ for $n=1,2,\cdots,\dfrac{m}{\rho}$
and 
\begin{equation}
\left|(\mathcal{L}\Phi^{N})(y)\right|\leq CN^{2-m-\rho}\mathcal{P}(Ny_{3})e^{-Ny_{3}},\,\,y\in\Omega_{N}\label{eq:Decay Estimate for approx. Sol}
\end{equation}
for some constant $C=C(m)>0$. Here $\mathcal{P}(Ny_{3})$ is a polynomial
with non-negative coefficients.\end{lem}
\begin{proof}
We look for $\Phi^{N}=(\Phi_{1}^{N},\Phi_{2}^{N},\Phi_{3}^{N})\in\mathbb{C}^{3}$
of the form 
\begin{equation}
\Phi_{k}^{N}(y)=e^{\sqrt{-1}Ny'\cdot\omega'}\mathbb{V}_{k}(N^{1-\rho}y',Ny_{3}).\label{eq:Appro Sol}
\end{equation}
Then 
\begin{align}
 & (\mathcal{L}\Phi^{N})_{i}\nonumber \\
= & \sum_{j,k,l=1}^{3}\dfrac{\partial}{\partial y_{j}}(\dot{C}_{ijkl}\dfrac{\partial}{\partial y_{l}}\Phi_{k}^{N})\nonumber \\
= & \sum_{k=1}^{3}\Big\{\sum_{j,l=1}^{2}\dot{C}_{ijkl}(\dfrac{\partial^{2}}{\partial y_{j}\partial y_{l}}\Phi_{k}^{N})+\sum_{j=1}^{2}\dot{C}_{ijk3}(\dfrac{\partial^{2}}{\partial y_{j}\partial y_{3}}\Phi_{k}^{N})+\sum_{l=1}^{2}\dot{C}_{i3kl}(\dfrac{\partial^{2}}{\partial y_{3}\partial y_{l}}\Phi_{k}^{N})\nonumber \\
 & +\dot{C}_{i3k3}\dfrac{\partial^{2}}{\partial y_{3}^{2}}\Phi_{k}^{N}+\sum_{j=1}^{3}\sum_{l=1}^{2}\big(\dfrac{\partial}{\partial y_{j}}\dot{C}_{ijkl}\big)\dfrac{\partial}{\partial y_{l}}\Phi_{k}^{N}+\sum_{j=1}^{3}\big(\dfrac{\partial}{\partial y_{j}}\dot{C}_{ijk3}\big)\dfrac{\partial}{\partial y_{3}}\Phi_{k}^{N}\Big\}.\label{eq:Expansion 1}
\end{align}
Substituting (\ref{eq:Appro Sol}) into (\ref{eq:Expansion 1}), we
have 
\begin{align}
 & (\mathcal{L}\Phi^{N})_{i}\nonumber \\
= & e^{\sqrt{-1}Ny'\cdot\omega'}\sum_{k=1}^{3}\Big[-N^{2}\sum_{j,l=1}^{2}\dot{C}_{ijkl}\omega_{j}\omega_{l}+\sqrt{-1}N\Big\{\sum_{j,l=1}^{2}\dot{C}_{ijkl}(\omega_{j}\dfrac{\partial}{\partial y_{l}}+\omega_{l}\dfrac{\partial}{\partial y_{j}})\nonumber \\
 & +\sum_{j=1}^{2}\dot{C}_{ijk3}\omega_{j}\dfrac{\partial}{\partial y_{3}}+\sum_{l=1}^{2}\dot{C}_{i3kl}\omega_{l}\dfrac{\partial}{\partial y_{3}}+\sum_{j,l=1}^{2}\big(\dfrac{\partial}{\partial y_{j}}\dot{C}_{ijkl}\big)\omega_{l}\Big\}\nonumber \\
 & +\sum_{j=1}^{2}\dot{C}_{ijk3}\dfrac{\partial^{2}}{\partial y_{j}\partial y_{3}}+\sum_{l=1}^{2}\dot{C}_{i3kl}\dfrac{\partial^{2}}{\partial y_{l}\partial y_{3}}+\sum_{j,l=1}^{2}\dot{C}_{ijkl}\dfrac{\partial^{2}}{\partial y_{j}\partial y_{l}}+\dot{C}_{i3k3}\dfrac{\partial^{2}}{\partial y_{3}^{2}}\nonumber \\
 & +\sum_{j=1}^{3}\sum_{l=1}^{2}\big(\dfrac{\partial}{\partial y_{j}}\dot{C}_{ijkl}\big)\dfrac{\partial}{\partial y_{l}}+\sum_{j=1}^{3}\big(\dfrac{\partial}{\partial y_{j}}\dot{C}_{ijk3}\big)\dfrac{\partial}{\partial y_{3}}\Big]\mathbb{V}_{k}(N^{1-\rho}y',Ny_{3}).\label{eq:Expansion 2}
\end{align}
Now, we introduce the scaled variables 
\begin{equation}
z_{i}=N^{1-\rho}y_{i}\mbox{ for }i=1,2\mbox{ and }z_{3}=Ny_{3},\label{eq:scaling}
\end{equation}
which implies 
\begin{equation}
\dfrac{\partial}{\partial y_{i}}=N^{1-\rho}\dfrac{\partial}{\partial z_{i}}\mbox{ for }i=1,2\mbox{ and }\dfrac{\partial}{\partial y_{3}}=N\dfrac{\partial}{\partial z_{3}}.
\end{equation}
We refer \eqref{eq:scaling} by \textit{scaling}. Then (\ref{eq:Expansion 2})
becomes 
\begin{align}
 & (\mathcal{L}\Phi^{N})_{i}\nonumber \\
= & e^{\sqrt{-1}Ny'\cdot\omega'}\sum_{k=1}^{3}\left\{ \left[N^{2}(\dot{C}_{i3k3}\dfrac{\partial^{2}}{\partial z_{3}^{2}}+\sqrt{-1}(\sum_{j=1}^{2}\dot{C}_{ijk3}\omega_{j}\dfrac{\partial}{\partial z_{3}}+\sum_{l=1}^{2}\dot{C}_{i3kl}\omega_{l}\dfrac{\partial}{\partial z_{3}}\right.\right.\nonumber \\
 & -\sum_{j,l=1}^{2}\dot{C}_{ijkl}\omega_{j}\omega_{l})+N^{2-\rho}(\sqrt{-1}\sum_{j,l=1}^{2}\dot{C}_{ijkl}(\omega_{j}\dfrac{\partial}{\partial z_{l}}+\omega_{l}\dfrac{\partial}{\partial z_{j}})\nonumber \\
 & +\sum_{j=1}^{2}\dot{C}_{ijk3}\dfrac{\partial^{2}}{\partial z_{j}\partial z_{3}}+\sum_{l=1}^{2}\dot{C}_{i3kl}\dfrac{\partial^{2}}{\partial z_{l}\partial z_{3}})+N^{2-2\rho}\sum_{j,l=1}^{2}\dot{C}_{ijkl}\dfrac{\partial^{2}}{\partial z_{j}\partial z_{l}}\nonumber \\
 & +N[\sqrt{-1}\sum_{j,l=1}^{2}\left(\dfrac{\partial}{\partial y_{j}}\dot{C}_{ijkl}\right)\omega_{l}+\sum_{j=1}^{3}\left(\dfrac{\partial}{\partial y_{j}}\dot{C}_{ijk3}\right)\dfrac{\partial}{\partial z_{3}}]\nonumber \\
 & \left.\left.+N^{1-\rho}\sum_{j=1}^{3}\sum_{l=1}^{2}\left(\dfrac{\partial}{\partial y_{j}}\dot{C}_{ijkl}\right)\dfrac{\partial}{\partial z_{l}}\right]\mathbb{V}_{k}(z',z_{3})\right\} .\label{eq:Expansion 3}
\end{align}
On the other hand, expand $\dot{C}_{ijkl}(y)$ of and $\dfrac{\partial}{\partial y_{n}}\dot{C}_{ijkl}(y)$
for $n=1,2,3$ into Taylor's series around $y=0$. Let $\beta\in(\mathbb{N}\cup\{0\})^{3}$
be a multi-index, then we have 
\[
\dot{C}_{ijkl}(y)=\sum_{|\beta|\leq m}\dfrac{1}{\beta!}\dfrac{\partial^{\beta}}{\partial y^{\beta}}C_{ijkl}(0)y^{\beta}+O(|y|^{m+p}),
\]
and for $n=1,2,3$, 
\[
\dfrac{\partial}{\partial y_{n}}\dot{C}_{ijkl}(y)=\sum_{|\beta|\leq m-1}\dfrac{1}{\beta!}\dfrac{\partial^{\beta}}{\partial y^{\beta}}\left(\dfrac{\partial}{\partial y_{n}}\dot{C}_{ijkl}\right)(0)y^{\mathbf{\beta}}+O(|y|^{m-1+p}).
\]
Recall that we have posed the condition (\ref{eq:smallness condition}),
which is $m+\rho\leq(1-\rho)(m+p).$ Thus, via (\ref{eq:scaling}),
we obtain 
\begin{equation}
\dot{C}_{ijkl}(y)=\sum_{|\beta|\leq m}N^{-\beta\cdot\alpha}\dfrac{1}{\beta!}\dfrac{\partial^{\beta}}{\partial z^{\beta}}\dot{C}_{ijkl}(0)z^{\beta}+R_{1}(z),\label{eq:Taylor 1}
\end{equation}
where 
\[
|R_{1}(z)|=O(N^{-m-\rho}).
\]
Similarly, for $n=1,2,3$, we have 
\begin{equation}
\dfrac{\partial}{\partial y_{n}}\dot{C}_{ijkl}(y)=\sum_{|\beta|\leq m}N^{-\beta\cdot\alpha}\dfrac{1}{\beta!}\dfrac{\partial^{\beta}}{\partial z^{\beta}}\left(\dfrac{\partial}{\partial y_{n}}\dot{C}_{ijkl}\right)(0)z^{\beta}+R_{2}(z),\label{eq:Taylor 2}
\end{equation}
where 
\[
|R_{2}(z)|=O(N^{-m-\rho+1}).
\]
Note that for the power of $N$ in the expansion (\ref{eq:Taylor 1})
and (\ref{eq:Taylor 2}) are of the form $-M\rho$ for some $M\in\mathbb{N}\cup\{0\}$.
Thus, we combine (\ref{eq:Expansion 3}), (\ref{eq:Taylor 1}), (\ref{eq:Taylor 2})
and $\mathbb{V}:=(\mathbb{V}_{1},\mathbb{V}_{2},\mathbb{V}_{3})^{t}$,
then 
\begin{equation}
\mathcal{L}\Phi^{N}=e^{\sqrt{-1}Ny'\cdot\omega'}\left[\sum_{s=0}^{\frac{m}{\rho}}N^{2-s\rho}L_{s}+L_{R}\right]\mathbb{V},\label{eq:Expansion 4}
\end{equation}
where $\mathcal{L}\Phi^{N}=((\mathcal{L}\Phi^{N})_{1},(\mathcal{L}\Phi^{N})_{2},(\mathcal{L}\Phi^{N})_{3})^{t}$
and $L_{s}$ ($s=0,1,2,\cdots,\frac{m}{\rho}$) are at most second
order matrix differential operators in $z'$ and $z_{3}$ with coefficients
depending on $y'$ and $y_{3}$. In particular $L_{0},\,L_{1},\,L_{2}$
are given by 
\begin{eqnarray*}
L_{0} & = & -\left(\dot{C}_{i3k3}(0)D_{3}^{2}+\left[\sum_{j=1}^{2}\dot{C}_{ijk3}(0)\omega_{j}\right.\right.\\
 &  & \left.\left.+\sum_{l=1}^{2}\dot{C}_{i3kl}(0)\omega_{l}\right]D_{3}+\sum_{j,l=1}^{2}\dot{C}_{ijkl}(0)\omega_{j}\omega_{l}\right)_{1\le i,k\le3},\\
L_{1} & = & \left(\sqrt{-1}\sum_{j,l=1}^{2}\dot{C}_{ijkl}(0)(\omega_{j}\dfrac{\partial}{\partial z_{l}}+\omega_{l}\dfrac{\partial}{\partial z_{j}})\right.\\
 &  & \left.+\sum_{j=1}^{2}\dot{C}_{ijk3}(0)\dfrac{\partial^{2}}{\partial z_{j}\partial z_{3}}+\sum_{l=1}^{2}\dot{C}_{i3kl}(0)\dfrac{\partial^{2}}{\partial z_{l}\partial z_{3}}\right)_{1\le i,k\le3},\\
L_{2} & = & \left(\sum_{j,l=1}^{2}\dot{C}_{ijkl}(0)\dfrac{\partial^{2}}{\partial z_{j}\partial z_{l}}\right)_{1\le i,k\le3},
\end{eqnarray*}
where $D_{3}=-\sqrt{-1}\dfrac{\partial}{\partial z_{3}}$. Moreover,
$L_{R}$ is a second order differential operator in $z$ and its coefficients
are of order $O(N^{2-m-\rho})$.

Now, we look for $\mathbb{V}(z',z_{3})$ of the form 
\begin{equation}
\mathbb{V}(z)=\sum_{n=0}^{\frac{m}{\rho}}N^{-\rho n}V^{n}(z),\label{V(z)}
\end{equation}
where $V^{n}(z)=(V_{1}^{n}(z),V_{2}^{n}(z),V_{3}^{n}(z))^{t}\in\mathbb{C}^{3}$.
By (\ref{eq:Expansion 4}) and equating the coefficients in each term
of order $N^{m}$, we have 
\begin{eqnarray}
\mathcal{L}\Phi^{N} & = & e^{\sqrt{-1}Ny'\cdot\omega'}\left[\left(\sum_{s=0}^{\frac{m}{\rho}}N^{2-s\rho}L_{s}\right)\left(\sum_{n=0}^{\frac{m}{\rho}}N^{-\rho n}V^{n}\right)+L_{R}\mathbb{V}\right]\label{eq:Expansion 5}\\
 & = & e^{\sqrt{-1}Ny'\cdot\omega'}\left[\sum_{r=0}^{\frac{m}{\rho}}N^{2-r\rho}\sum_{n+s=r}L_{s}V^{n}+\mathcal{R}\right],\nonumber 
\end{eqnarray}
where 
\[
\mathcal{R}:=\sum_{r=\frac{m}{\rho}+1}^{\frac{2m}{\rho}}N^{2-r\rho}\sum_{n+s=r}L_{s}V^{n}+L_{R}\mathbb{V}.
\]
Therefore, we have obtained the following ordinary differential equations
of systems (ODE systems) of second order with respect to $z_{3}$
\begin{align}
 & L_{0}V^{0}=0,\nonumber \\
 & L_{0}V^{1}+L_{1}V^{0}=0,\nonumber \\
 & L_{0}V^{2}+L_{1}V^{1}+L_{2}V^{0}=0,\label{eq:ODE systems}\\
 & \cdots\nonumber \\
 & L_{0}V^{\frac{m}{\rho}}+\cdots+L_{\frac{m}{\rho}}V^{0}=0,\nonumber 
\end{align}
with boundary conditions 
\begin{align*}
 & V^{0}|_{z_{3}=0}=\eta^{N}(y')\,\mathbf{a}=\eta(z')\,\mathbf{a},\\
 & V^{n}|_{z_{3}=0}=0\mbox{ for }n=1,2,\cdots,\dfrac{m}{\rho}.
\end{align*}
Note that this undetermined boundary value problem (\ref{eq:ODE systems})
can be made determined if we look for solutions which are bounded
in $z_{3}\in[0,\infty)$.

First, by using the Stroh formalism which is for instance given in
\cite{NUW2005(ODS)}, we can solve $L_{0}V^{0}=0$ with $V^{0}(z',0)=\eta(z')$
in the following way. Let $\xi=(\xi_{1},\xi_{2},\xi_{3})$, $\zeta=(\zeta_{1},\zeta_{2},\zeta_{3})\in\mathbb{R}^{3}$
and we define the $3\times3$ matrix $\left\langle \xi,\zeta\right\rangle $
by 
\[
\left\langle \xi,\zeta\right\rangle =(\left\langle \xi,\zeta\right\rangle _{ik})\,\,\text{with}\,\,\left\langle \xi,\zeta\right\rangle _{ik}=\sum_{1\le j,\,l\le3}\dot{C}_{ijkl}(y',y_{3})\xi_{j}\,\zeta_{l}.
\]
Also, if we set $\left\langle \xi,\zeta\right\rangle _{0}:=\left\langle \xi,\zeta\right\rangle |_{y=0}$
, $e_{3}:=(0,0,1)$ and $\omega:=(\omega_{1},\omega_{2},0)$, we can
rewrite 
\[
L_{0}V^{0}=-\left[\left\langle e_{3},e_{3}\right\rangle _{0}D_{3}^{2}+\left(\left\langle e_{3},\omega\right\rangle _{0}+\left\langle \omega,e_{3}\right\rangle _{0}\right)D_{3}+\left\langle \omega,\omega\right\rangle _{0}\right]V^{0}=0.
\]
Let $W_{1}^{0}=V^{0}$, $W_{2}^{0}=-\{\left\langle e_{3},e_{3}\right\rangle _{0}D_{3}V^{0}+\left\langle e_{3},\omega\right\rangle _{0}V^{0}\}$,
then by direct calculation, then we have 
\begin{equation}
D_{3}W_{1}^{0}=-\left\langle e_{3},e_{3}\right\rangle _{0}^{-1}\left[\left\langle e_{3},\omega\right\rangle _{0}W_{1}^{0}-W_{0}^{2}\right]\label{eq:W_1}
\end{equation}
and 
\begin{equation}
D_{3}W_{2}^{0}=\left[\left\langle \omega,\omega\right\rangle _{0}-\left\langle \omega,e_{3}\right\rangle _{0}\left\langle e_{3},e_{3}\right\rangle _{0}^{-1}\left\langle e_{3},\omega\right\rangle _{0}\right]W_{1}^{0}-\left\langle \omega,e_{3}\right\rangle _{0}\left\langle e_{3},e_{3}\right\rangle _{0}^{-1}W_{2}^{0}.\label{eq:W_2}
\end{equation}
Combine (\ref{eq:W_1}), (\ref{eq:W_2}) and define the column vector
$W^{0}:=[W_{1}^{0},W_{2}^{0}]$, then we obtain 
\begin{equation}
D_{3}W^{0}=K^{0}W^{0},\label{eq:First order}
\end{equation}
where 
\[
K^{0}=\left[\begin{array}{cc}
-\left\langle e_{3},e_{3}\right\rangle _{0}^{-1}\left\langle e_{3},\omega\right\rangle _{0} & -\left\langle e_{3},e_{3}\right\rangle _{0}^{-1}\\
-\left\langle \omega,\omega\right\rangle _{0}+\left\langle \omega,e_{3}\right\rangle _{0}\left\langle e_{3},e_{3}\right\rangle _{0}^{-1}\left\langle e_{3},\omega\right\rangle _{0} & -\left\langle \omega,e_{3}\right\rangle _{0}\left\langle e_{3},e_{3}\right\rangle _{0}^{-1}
\end{array}\right].
\]
Note that $K^{0}$ is a $6\times6$ matrix-valued function independent
of $z_{3}$ variable and its eigenvalues are determined by 
\begin{equation}
\det(\Sigma I_{6}-K^{0})=0,\label{eq:DET}
\end{equation}
where $I_{6}$ is a $6\times6$ identity matrix and (\ref{eq:DET})
is equivalent to 
\begin{equation}
\det\left[\left\langle e_{3},e_{3}\right\rangle _{0}\Sigma^{2}+\left(\left\langle e_{3},\omega\right\rangle _{0}+\left\langle \omega,e_{3}\right\rangle _{0}\right)\Sigma+\left\langle \omega,\omega\right\rangle _{0}\right]=0.\label{eq:DET2}
\end{equation}
By using the results of \cite[Chapter 1.8]{tanuma2007stroh} , we
have 
\begin{align*}
 & \det\left[\left\langle e_{3},e_{3}\right\rangle _{0}\Sigma^{2}+\left(\left\langle e_{3},\omega\right\rangle _{0}+\left\langle \omega,e_{3}\right\rangle _{0}\right)\Sigma+\left\langle \omega,\omega\right\rangle _{0}\right]\\
= & \mu^{2}(0)(\lambda+2\mu)(0)(1+\Sigma^{2})^{3},
\end{align*}
which means solving (\ref{eq:DET}) is equivalent to solve 
\[
(1+\Sigma^{2})^{3}=0
\]
and use the strong convexity condition (\ref{eq:Strong Convex}),
then it gives that the roots are $\Sigma=\pm\sqrt{-1}$. Moreover,
we can find eigenvectors $\{\widetilde{q_{1}^{+}},\widetilde{q_{2}^{+}},\widetilde{q_{3}^{+}},\widetilde{q_{1}^{-}},\widetilde{q_{2}^{-}},\widetilde{q_{3}^{-}}\}(0)$
of $K^{0}$, i.e., 
\begin{equation}
K^{0}\widetilde{q_{\gamma}^{\pm}}=\pm\sqrt{-1}\widetilde{q_{\gamma}^{\pm}}\mbox{ for }\gamma=1,2,3,\label{eq:Stroh eigenvalue problem}
\end{equation}
with $\widetilde{q_{\gamma}^{+}}$ being the complex conjugate of
$\widetilde{q_{\gamma}^{-}}$, or $\widetilde{q_{\gamma}^{+}}=\overline{\widetilde{q_{\gamma}^{-}}}$
at $y=0$.

According to the result in \cite{tanuma2007stroh}, the eigenvalue
problem for $K^{0}$ is \textit{degenerate} and there are generalized
eigenvectors. More precisely, let 
\[
\widetilde{q_{1}^{+}}=\left(\begin{array}{c}
\omega_{2}\\
-\omega_{1}\\
0\\
\sqrt{-1}\mu\omega_{2}\\
-\sqrt{-1}\mu\omega_{1}\\
0
\end{array}\right)(0),\mbox{ }\widetilde{q_{2}^{+}}=\left(\begin{array}{c}
\omega_{1}\\
\omega_{2}\\
\sqrt{-1}\\
-2\mu\omega_{1}\\
-2\mu\omega_{2}\\
2\sqrt{-1}
\end{array}\right)(0)
\]
and 
\[
\widetilde{q_{3}^{+}}=\left(\begin{array}{c}
0\\
0\\
-\frac{\lambda+3\mu}{\lambda+\mu}\\
-\frac{2\mu^{2}}{\lambda+\mu}\omega_{1}\\
-\frac{2\mu^{2}}{\lambda+\mu}\omega_{2}\\
-\sqrt{-1}\frac{2\mu(\lambda+2\mu)}{\lambda+\mu}
\end{array}\right)(0)
\]
such that 
\[
K^{0}\widetilde{q_{3}^{+}}-\sqrt{-1}\widetilde{q_{3}^{+}}=\widetilde{q_{2}^{+}}.
\]
and define 
\[
\widetilde{Q}:=\left(\widetilde{q_{1}^{+}},\widetilde{q_{2}^{+}},\widetilde{q_{3}^{+}},\widetilde{q_{1}^{-}},\widetilde{q_{2}^{-}},\widetilde{q_{3}^{-}}\right),
\]
which is a non-singular matrix giving the Jordan canonical form 
\[
\widetilde{Q}^{-1}K^{0}\widetilde{Q}=\left(\begin{array}{cccccc}
\sqrt{-1}\\
 & \sqrt{-1} & 1\\
 &  & \sqrt{-1}\\
 &  &  & -\sqrt{-1}\\
 &  &  &  & -\sqrt{-1} & 1\\
 &  &  &  &  & -\sqrt{-1}
\end{array}\right).
\]
Since we want to have a general form of solution of \eqref{eq:First order}
which is bounded for $z_{3}\in[0,\infty)$, we take $\zeta=\sqrt{-1}$.
Further we take linearly independent vectors $\sigma_{1}=\left(\begin{array}{c}
\omega_{2}\\
-\omega_{1}\\
0
\end{array}\right)$, $\sigma_{2}=\left(\begin{array}{c}
\omega_{1}\\
\omega_{2}\\
\sqrt{-1}
\end{array}\right)$ and $\sigma_{3}=\left(\begin{array}{c}
0\\
0\\
-\frac{\lambda+3\mu}{\lambda+\mu}(0)
\end{array}\right)$. Then, for any given $\mathbf{a}=(a_{1},a_{2},a_{3})$, there exists
constants $c_{\beta}\in\mathbb{C}$ ($\beta=1,2,3$) such that $\mathbf{a}=\sum_{s=1}^{3}c_{s}\sigma_{s}$.
Therefore, as in \cite[Lemma 1.6 and (2.66)]{tanuma2007stroh}, $V^{0}(z',z_{3})$
is given as 
\begin{eqnarray}
V^{0}(z',z_{3}) & = & e^{-z_{3}}\eta(z')\big(\sum_{s=1}^{3}c_{s}\sigma_{s}-\sqrt{-1}c_{3}\,\sigma_{2}\,z_{3}\big)\label{eq:V^0}\\
 & = & e^{-z_{3}}\eta(z')\,\big(\mathbf{a}-\sqrt{-1}c_{3}\,z_{3}\,\sigma_{2}\big),\nonumber 
\end{eqnarray}
which is a $C^{\infty}$-smooth solution of $L_{0}V^{0}=0$ with $V^{0}(z',0)=\eta(z')\,\mathbf{a}$.

Next, we solve 
\begin{equation}
L_{0}V^{1}+L_{1}V^{0}=0\mbox{ with }V^{1}(z',0)=0.\label{eq:V^1}
\end{equation}
Since 
\begin{eqnarray*}
L_{1} & = & \big(\sqrt{-1}\sum_{j,l=1}^{2}\dot{C}_{ijkl}(0)(\omega_{j}\dfrac{\partial}{\partial z_{l}}+\omega_{l}\dfrac{\partial}{\partial z_{j}})\\
 &  & +\sum_{j=1}^{2}\dot{C}_{ijk3}(0)\dfrac{\partial^{2}}{\partial z_{j}\partial z_{3}}+\sum_{l=1}^{2}\dot{C}_{i3kl}(0)\dfrac{\partial^{2}}{\partial z_{l}\partial z_{3}}\big)_{1\le i,k\le3},
\end{eqnarray*}
and 
\[
L_{1}V^{0}(z',z_{3})=e^{-z_{3}}\sum_{d=0}^{1}P_{0}^{d}(z')z_{3}^{d},
\]
where $P_{0}^{d}(z')$ are $C^{\infty}$-smooth vector-valued function
depending on $\partial_{z'}^{\beta}\eta(z')$ for multi-indices $|\beta|\leq1$
for $d=0,1$. It is worth mentioning the following observation. That
is for any $d\in\mathbb{N}$, we have by direct computation 
\[
L_{0}\left(z_{3}^{d}e^{-z_{3}}\right)=e^{-z_{3}}\left(z_{3}^{d-1}\,R_{1}^{d}+z_{3}^{d-2}R_{2}^{d}\right)
\]
with invertible matrices 
\begin{eqnarray*}
R_{1}^{d} & = & d\left[2\left\langle e_{3},e_{3}\right\rangle _{0}-\sqrt{-1}\left(\left\langle \omega,e_{3}\right\rangle _{0}+\left\langle e_{3},\omega\right\rangle _{0}\right)\right],\\
R_{2}^{d} & = & -d(d-1)\left\langle e_{3},e_{3}\right\rangle _{0}.
\end{eqnarray*}
Based on this we look for $V^{1}(z',z_{3})$ in the following form
\begin{eqnarray}
V^{1}(z',z_{3}) & = & e^{-z_{3}}\sum_{d=1}^{2}z_{3}^{d}P_{1}^{d}(z'),\label{eq:V^1-1}
\end{eqnarray}
where $P_{1}^{d}(z')\in\mathbb{C}^{3}$ are vector-valued functions
which will be determined later. By straightforward computation, we
can have 
\begin{equation}
L_{0}V^{1}=z_{3}\,e^{-z_{3}}\,R_{1}^{2}+e^{-z_{3}}\big\{ R_{2}^{2}P_{1}^{2}(z')+R_{1}^{1}P_{1}^{1}(z')\big\}.
\end{equation}
Then by equating the equation $L_{0}V^{1}=-L_{1}V^{0}$, we have 
\begin{align}
 & P_{0}^{1}(z')=-R_{1}^{2}P_{1}^{2}(z'),\label{eq:P_0 1 (z')}\\
 & P_{0}^{0}(z')=-R_{2}^{2}\,P_{1}^{2}(z')-R_{1}^{1}\,P_{1}^{1}(z').\label{eq:P_0 2 (z')}
\end{align}
In order to solve $P_{1}^{d}(z')$ explicitly for $d=1,2$, first,
we can invert the right hand side of (\ref{eq:P_0 1 (z')}) to find
$P_{1}^{2}(z')$ and plug it into (\ref{eq:P_0 2 (z')}) to know $P_{1}^{1}(z')$.
Thus we have (\ref{eq:V^1-1}).

Further for each $n\ge2$, we can express the solution $V^{n}(z',z_{3})$
of \eqref{eq:ODE systems} with $V^{n}(z',0)=0$ inductively as 
\begin{eqnarray*}
V^{n}(z) & = & \sum_{d=1}^{n+1}z_{3}^{d}P_{n}^{d}(z')e^{-z_{3}},
\end{eqnarray*}
where $P_{n}^{d}(z')$ are smooth vector-valued functions depending
on $\mu(0),\lambda(0),\eta(z')$ and supported in $\{|z'|<1\}$ for
$n=1,2,\cdots,\dfrac{m}{\rho}$. Finally, from \eqref{eq:Expansion 5}
and (\ref{eq:ODE systems}), we have 
\[
\mathcal{L}\Phi^{N}=L_{R}\mathbb{V}.
\]
Here note that the coefficients of $L_{R}$ are of the form $O(N^{2-m-\rho})$
multiplied with polynomials in $z_{3}$. Therefore, there exists $C=C(m)$
such that 
\[
\left|L_{R}\mathbb{V}\right|\leq CN^{2-m-\rho}\mathcal{P}(z_{3})e^{-z_{3}},
\]
where $\mathcal{P}(z_{3})$ is a polynomial in $z_{3}$ with non-negative,
which completes the proof. 
\end{proof}

\section{Proof of Theorem \ref{Main Thm}, (2)\label{Section 3}}

In this section, we prove item (2) of Theorem 1.1. Our ideas are initiated
from \cite{kang2002boundary,nakamura2001local}. Let $\zeta(y_{3})\in C^{\infty}([0,\infty))$
satisfy $0\leq\zeta\leq1$, $\zeta(y_{3})=1$ for $0\leq y_{3}\leq\dfrac{1}{2}$
and $\zeta(y_{3})=0$ for $y_{3}\geq1$ and put 
\[
\zeta_{N}(y_{3})=\zeta(\sqrt{N}y_{3}).
\]
Given $\epsilon>0$, choose large $N\in\mathbb{N}$, 
\[
\mbox{supp}(\eta^{N}\zeta_{N})\subset\Omega_{\epsilon}:=\{|x|\leq\epsilon\}.
\]

For $m\in\mathbb{N}$, recall that the regularity of $\partial\Omega$
is of $C^{m+2}$ class. For convenience, denote $\mathcal{C}^{m}=\mathcal{C}^{m,0}$,
$\lambda^{m}=\lambda^{m,0}$ and $\mu^{m}=\mu^{m,0}$, where $\mathcal{C}^{m,x}$,
$\lambda^{m,x}$ and $\mu^{m,x}$ were introduced in (\ref{eq:C_bx})
and (\ref{eq:Lame B}). Let $u^{N}=(u_{1}^{N},u_{2}^{N},u_{3}^{N})\in H^{1}(\Omega;\mathbb{C}^{3})$
be the solution to 
\begin{equation}
\begin{cases}
\sum_{j,k,l=1}^{3}\dfrac{\partial}{\partial x_{j}}(\dot{C}_{ijkl}\dfrac{\partial}{\partial x_{l}}u_{k}^{N})=0\,\,\,(1\le i\le3)\,\, & \mbox{ in }\Omega,\\
u^{N}=\phi^{N} & \mbox{ on }\partial\Omega,
\end{cases}\label{uN}
\end{equation}
\label{vN} and let $v^{N}=(v_{1}^{N},v_{2}^{N},v_{3}^{N})\in H^{1}(\Omega;\mathbb{C}^{3})$
be the solution to 
\begin{equation}
\begin{cases}
\sum_{j,k,l=1}^{3}\dfrac{\partial}{\partial x_{j}}(\dot{C}_{ijkl}^{m}\dfrac{\partial}{\partial x_{l}}v_{k}^{N})=0\,\,\,(1\le i\le3)\,\  & \mbox{ in }\Omega,\\
v^{N}=\phi^{N} & \mbox{ on }\partial\Omega.
\end{cases}
\end{equation}
Let $\zeta_{N}\,\Phi^{N}$ and $\zeta_{N}\,\Psi^{N}$ be approximate
solutions of $u^{N}$ and $v^{N}$ with $\zeta_{N}\,\Phi^{N}|_{\partial\Omega}=\zeta_{N}\,\Psi^{N}|_{\partial\Omega}=\phi^{N}$,
respectively. Likewise the construction in Section 2, we can express
$\Psi^{N}$ as 
\begin{equation}
\Psi^{N}(y)=e^{\sqrt{-1}Ny'\cdot\omega'}e^{-Ny_{3}}\left\{ \eta^{N}(y')\,\mathbf{a}+\sum_{n=1}^{\frac{m}{\rho}}N^{-n\rho}v_{n}^{m}(N^{\alpha}y)\right\} ,\label{eq:Psi^N}
\end{equation}
where $v_{n}^{m}(N^{\alpha}y)$ are polynomials in $Ny_{3}$ depending
on $\dot{C}_{ijkl}^{m}(0)$ and their coefficients are $C^{\infty}$-smooth
functions of $N^{\alpha}y$ supported in $\{|y'|<N^{\rho-1}\}$ for
$n=1,2,\cdots,\dfrac{m}{\rho}$.

Note that $\left\langle \Lambda_{\mathcal{C}^{m}}\phi^{N},\overline{\phi^{N}}\right\rangle $
is real and hence we have 
\[
\left\langle \Lambda_{\mathcal{C}^{m}}\phi^{N},\overline{\phi^{N}}\right\rangle =\left\langle \overline{\Lambda_{\mathcal{C}^{m}}\phi^{N}},\phi^{N}\right\rangle .
\]
By direct calculation, we have 
\begin{alignat*}{1}
 & \left\langle (\Lambda_{\mathcal{C}}-\Lambda_{\mathcal{C}^{m}})\phi^{N},\overline{\phi^{N}}\right\rangle \\
= & \int_{\Omega}\left[\lambda\mbox{div}\,u^{N}(\overline{\mbox{div}\,(\zeta^{N}\Psi^{N})})+2\mu\epsilon(u^{N}):\overline{\epsilon(\zeta^{N}\Psi^{N})}\right]dy\\
 & -\int_{\Omega}\left[\lambda^{m}\mbox{div}\,v^{N}(\overline{\mbox{div}\,(\zeta^{N}\Phi^{N})})+2\mu^{m}\epsilon(v^{N}):\overline{\epsilon(\zeta^{N}\Phi^{N})}\right]dy.
\end{alignat*}
Let 
\[
u^{N}=\Phi^{N}+f^{N}\mbox{ and }v^{N}=\Psi^{N}+g^{N}
\]
with 
\[
f^{N}|_{\partial\Omega}=g^{N}|_{\partial\Omega}=0.
\]
Then we have 
\begin{align*}
 & \left\langle (\Lambda_{\mathcal{C}}-\Lambda_{\mathcal{C}^{m}})\phi^{N},\overline{\phi^{N}}\right\rangle \\
= & \int_{\Omega}\left[\lambda\mbox{div}\,\Phi^{N}(\overline{\mbox{div}\,(\zeta^{N}\Psi^{N})})+2\mu\epsilon(\Phi^{N}):\overline{\epsilon(\zeta^{N}\Psi^{N})}\right]dy\\
 & -\int_{\Omega}\left[\lambda^{m}\mbox{div}\,\Psi^{N}(\overline{\mbox{div}\,(\zeta^{N}\Phi^{N})}+2\mu^{m}\epsilon(\Psi^{N}):\overline{\epsilon(\zeta^{N}\Phi^{N})}\right]dy\\
 & +\int_{\Omega}\left[\lambda\mbox{div}\,f^{N}(\overline{\mbox{div}\,(\zeta^{N}\Psi^{N})})+2\mu\epsilon(f^{N}):\overline{\epsilon(\zeta^{N}\Psi^{N})}\right]dy\\
 & -\int_{\Omega}\left[\lambda^{m}\mbox{div}\,g^{N}(\overline{\mbox{div}\,(\zeta^{N}\Phi^{N})})+2\mu^{m}\epsilon(g^{N}):\overline{\epsilon(\zeta^{N}\Phi^{N})}\right]dy\\
:= & I+II+III,
\end{align*}
where 
\begin{eqnarray*}
I & = & \int_{\Omega}\left[\lambda\mbox{div}\,\Phi^{N}(\overline{\mbox{div}\,(\zeta^{N}\Psi^{N})})+2\mu\epsilon(\Phi^{N}):\overline{\epsilon(\zeta^{N}\Psi^{N})}\right]dy\\
 &  & -\int_{\Omega}\left[\lambda^{m}\mbox{div}\,\Psi^{N}(\overline{\mbox{div}\,(\zeta^{N}\Phi^{N})})+2\mu^{m}\epsilon(\Psi^{N}):\overline{\epsilon(\zeta^{N}\Phi^{N})}\right]dy,\\
II & = & \int_{\Omega}\left[\lambda\mbox{div}\,f^{N}(\overline{\mbox{div}\,(\zeta^{N}\Psi^{N})})+2\mu\epsilon(f^{N}):\overline{\epsilon(\zeta^{N}\Psi^{N})}\right]dy,\\
III & = & -\int_{\Omega}\left[\lambda^{m}\mbox{div}\,g^{N}(\overline{\mbox{div}\,(\zeta^{N}\Phi^{N})})+2\mu^{m}\epsilon(g^{N}):\overline{\epsilon(\zeta^{N}\Phi^{N})}\right]dy.
\end{eqnarray*}
We will estimate $I$, $II$ and $III$ separately in the next subsections.

\subsection{Estimate of $I$}

Let 
\[
\Omega_{N}':=\left\{ y:|y_{1}|,|y_{2}|\leq N^{\rho-1},\mbox{ }\dfrac{1}{2\sqrt{N}}\leq y_{3}\leq\dfrac{1}{\sqrt{N}}\right\} \mbox{ and }D_{N}:=\Omega_{N}\backslash\Omega_{N}'.
\]
then we can rewrite $I$ as 
\begin{eqnarray*}
I & = & \int_{D_{N}}\left[(\lambda-\lambda^{m})\mbox{div}\,\Phi^{N}(\overline{\mbox{div}\,\Psi^{N}})+2(\mu-\mu^{m})\epsilon(\Phi^{N}):\overline{\epsilon(\Psi^{N})}\right]dy\\
 &  & +\int_{\Omega_{N}'}\left[\lambda\mbox{div}\,\Phi^{N}(\overline{\mbox{div}\,(\zeta^{N}\Psi^{N})})+2\mu\epsilon(\Phi^{N}):\overline{\epsilon(\zeta^{N}\Psi^{N})}\right]dy\\
 &  & -\int_{\Omega_{N}'}\left[\lambda^{m}\mbox{div}\,\Psi^{N}(\overline{\mbox{div}\,(\zeta^{N}\Phi^{N})})+2\mu^{m}\epsilon(\Psi^{N}):\overline{\epsilon(\zeta^{N}\Phi^{N})}\right]dy\\
 & := & I_{1}+I_{2},
\end{eqnarray*}
where 
\begin{eqnarray*}
I_{1} & = & \int_{D_{N}}(\lambda-\lambda^{m})\mbox{div}\,\Phi^{N}(\overline{\mbox{div}\,\Psi^{N}})+2(\mu-\mu^{m})\epsilon(\Phi^{N}):\overline{\epsilon(\Psi^{N})}dy,\\
I_{2} & = & \int_{\Omega_{N}'}\lambda\mbox{div}\,\Phi^{N}(\overline{\mbox{div}\,(\zeta^{N}\Psi^{N})})+2\mu\epsilon(\Phi^{N}):\overline{\epsilon(\zeta^{N}\Psi^{N})}dy,\\
 &  & -\int_{\Omega_{N}'}\lambda^{m}\mbox{div}\,\Psi^{N}(\overline{\mbox{div}\,(\zeta^{N}\Phi^{N})})-2\mu^{m}\epsilon(\Psi^{N}):\overline{\epsilon(\zeta^{N}\Phi^{N})}dy.
\end{eqnarray*}
Recall that $\mathbf{a}=(a_{1},a_{2},a_{3})\in\mathbb{C}^{3}$, then
for $I_{1}$, by (\ref{eq:Approximate Solution}), (\ref{eq:Psi^N}),
and direct calculation, we have 
\[
\mbox{div}\,\Phi^{N}=Ne^{\sqrt{-1}Ny'\cdot\omega'}e^{-Ny_{3}}\left(\sqrt{-1}\sum_{i=1}^{2}\omega_{i}a_{i}-a_{3}\right)\eta^{N}(y')+O(N^{1-\rho})e^{-c_{0}Ny_{3}},
\]
with some constant $c_{0}>0$. Hereafter $c_{0}$ denotes a general
constant which may differ time to time. Also for each $\epsilon_{ij}(\Phi^{N})$
of $\epsilon(\Phi^{N})=(\epsilon_{ij}(\Phi^{N}))$ we have 
\[
\begin{array}{ll}
\epsilon_{ij}(\Phi^{N}) & =\sqrt{-1}Ne^{\sqrt{-1}Ny'\cdot\omega'}e^{-Ny_{3}}\dfrac{a_{i}\omega_{j}+a_{j}\omega_{i}}{2}\eta^{N}(y')\\
 & \qquad\qquad+O(N^{1-\rho})e^{-c_{0}Ny_{3}},\\
\epsilon_{i3}(\Phi^{N}) & =Ne^{\sqrt{-1}Ny'\cdot\omega'}e^{-Ny_{3}}\dfrac{\sqrt{-1}a_{3}\omega_{i}-a_{i}}{2}\eta^{N}(y')+O(N^{1-\rho})e^{-c_{0}Ny_{3}},\\
\epsilon_{33}(\Phi^{N}) & =-Ne^{\sqrt{-1}Ny'\cdot\omega'}e^{-Ny_{3}}\eta^{N}(y')\,a_{3}+O(N^{1-\rho})e^{-c_{0}Ny_{3}}.
\end{array}
\]
Similarly, we have 
\[
\mbox{div}\,\Psi^{N}=Ne^{\sqrt{-1}Ny'\cdot\omega'}e^{-Ny_{3}}\left(\sqrt{-1}\sum_{i=1}^{2}\omega_{i}\,a_{i}-a_{3}\right)\eta^{N}(y')+O(N^{1-\rho})e^{-c_{0}Ny_{3}}
\]
and for $i,j=1,2$, we have 
\[
\begin{array}{ll}
\epsilon_{ij}(\Psi^{N}) & =\sqrt{-1}Ne^{\sqrt{-1}Ny'\cdot\omega'}e^{-Ny_{3}}\dfrac{a_{i}\omega_{j}+a_{j}\omega_{i}}{2}\eta^{N}(y')\\
 & \qquad\qquad+O(N^{1-\rho})e^{-c_{0}Ny_{3}},\\
\epsilon_{i3}(\Psi^{N}) & =Ne^{\sqrt{-1}Ny'\cdot\omega'}e^{-Ny_{3}}\dfrac{\sqrt{-1}a_{3}\omega_{i}-a_{i}}{2}\eta^{N}(y')+O(N^{1-\rho})e^{-c_{0}Ny_{3}},\\
\epsilon_{33}(\Psi^{N}) & =Ne^{\sqrt{-1}Ny'\cdot\omega'}e^{-Ny_{3}}\eta^{N}(y')+O(N^{1-\rho})e^{-c_{0}Ny_{3}}.
\end{array}
\]
Recall that 
\[
\epsilon(\Phi^{N}):\epsilon(\Psi^{N})=\sum_{1\leq i,j\leq3}\epsilon_{ij}(\Phi^{N})\epsilon_{ij}(\Psi^{N}),
\]
by straightforward calculation, then we have 
\begin{eqnarray*}
I_{1} & = & N^{2}\int_{0}^{\frac{1}{2\sqrt{N}}}\int_{|y'|\leq N^{\rho-1}}e^{-2Ny_{3}}\Big\{\eta^{N}(y')^{2}(\lambda-\lambda^{m})\\
 &  & \times\Big(\sqrt{-1}\sum_{i=1}^{2}\omega_{i}a_{i}-a_{3}\Big)^{2}+2\Big(\mu-\mu^{m}\Big)\eta^{N}(y')^{2}\\
 &  & \times\left[\sum_{i,j=1}^{2}\Big(\dfrac{a_{i}\omega_{j}+a_{j}\omega_{i}}{2}\Big)^{2}+2\sum_{i=1}^{2}\Big(\dfrac{\sqrt{-1}a_{3}\omega_{i}-a_{i}}{2}\Big)^{2}+a_{3}^{2}\right]\Big\} dy'\,dy_{3}\\
 &  & +O(N^{2-\rho})\int_{0}^{\frac{1}{2\sqrt{N}}}\int_{|y'|\le N^{\rho-1}}e^{-2c_{0}Ny_{3}}\big(|\lambda-\lambda^{m}|+|\mu-\mu^{m}|\big)\,dy'\,dy_{3}.
\end{eqnarray*}

For further argument we need the following lemma which was proved
in \cite{nakamura2001local}. 
\begin{lem}
\cite{nakamura2001local} For any $k\in\mathbb{N}$, $f(y)=f(y',y_{3})\in C^{k}$
around $y=0$, we have 
\begin{align}
 & \lim_{N\to\infty}N^{2+k}\int_{0}^{\frac{1}{2\sqrt{N}}}\int_{|y'|\leq N^{\rho-1}}\eta^{N}(y')^{2}e^{-2Ny_{3}}\left(f(y)-f^{k}(y)\right)dy'dy_{3}\nonumber \\
 & \qquad\qquad=\dfrac{1}{2^{k+1}}\dfrac{\partial^{k}f}{\partial y_{3}^{k}}(0),\label{eq:Main Limit}
\end{align}
where $f^{k}(y)=\sum_{n=0}^{k-1}\dfrac{1}{n!}\dfrac{\partial^{n}}{\partial y_{3}^{n}}f(y',0)y_{3}^{n}$.
Also for each $d\in\mathbb{N}$, we have 
\begin{align}
 & \lim_{N\to\infty}N^{2-\rho+k}\int_{0}^{\frac{1}{2\sqrt{N}}}\int_{|y'|\leq N^{\rho-1}}\psi(\sqrt{N}y')(Ny_{3})^{d}e^{-2Ny_{3}}\Big|f(y)-f^{k}(y)\Big|\,dy'dy_{3}\nonumber \\
 & \qquad\qquad\qquad=0,\label{eq:Not Important Limit}
\end{align}
where $\psi=\psi(\sqrt{N}y')$ is a $C^{\infty}$ function supported
in $\{|y'|<\dfrac{1}{\sqrt{N}}\}$. 
\end{lem}
By using (\ref{eq:Main Limit}) and (\ref{eq:Not Important Limit}),
we can see that 
\begin{align}
 & \lim_{N\to\infty}N^{m}I_{1}\label{eq:I_1}\\
= & \dfrac{1}{2^{m+1}}\dfrac{\partial^{m}\lambda}{\partial y_{3}^{m}}(0)\left(\sqrt{-1}\sum_{i=1}^{2}\omega_{i}a_{i}-a_{3}\right)^{2}\nonumber \\
 & +\dfrac{1}{2^{m}}\dfrac{\partial^{m}\mu}{\partial y_{3}^{m}}(0)\left[\sum_{i,j=1}^{2}\left(\dfrac{a_{i}\omega_{j}+a_{j}\omega_{i}}{2}\right)^{2}+2\left(\dfrac{\sqrt{-1}a_{3}\omega_{i}-a_{i}}{2}\right)^{2}+1\right].\nonumber 
\end{align}
For $I_{2}$, via (\ref{eq:Approximate Solution}) and (\ref{eq:Psi^N}),
we have $|\Phi^{N}|+|\Psi^{N}|\leq\exp(-c_{1}Ny_{3})$ for some constant
$c_{1}>0$ and 
\begin{equation}
|I_{2}|\le c_{0}\exp\left(-c_{1}\dfrac{N^{\frac{1}{2}}}{2}\right).\label{eq:I_2}
\end{equation}
Combining (\ref{eq:I_1}) and (\ref{eq:I_2}), we have 
\begin{align*}
 & \lim_{N\to\infty}N^{m}I\\
= & \dfrac{1}{2^{m+1}}\dfrac{\partial^{m}\lambda}{\partial y_{3}^{m}}(0)\left(\sqrt{-1}\sum_{i=1}^{2}\omega_{i}a_{i}-a_{3}\right)^{2}\\
 & +\dfrac{1}{2^{m}}\dfrac{\partial^{m}\mu}{\partial y_{3}^{m}}(0)\left[\sum_{i,j=1}^{2}\left(\dfrac{a_{i}\omega_{j}+a_{j}\omega_{i}}{2}\right)^{2}+2\left(\dfrac{\sqrt{-1}a_{3}\omega_{i}-a_{i}}{2}\right)^{2}+a_{3}^{2}\right].
\end{align*}

\subsection{Estimates of $II$ and $III$}

In this section, we will prove 
\[
\lim_{N\to\infty}N^{m}II=\lim_{N\to\infty}N^{m}III=0.
\]
Since we can use the same method to prove $II$ and $III$, we only
prove the case 
\begin{equation}
\lim_{N\to\infty}N^{m}II=0.\label{eq:II}
\end{equation}
By direct calculation, 
\begin{eqnarray*}
II & = & \int_{\Omega}\left[\lambda\mbox{div}\,f^{N}(\overline{\mbox{div}\,(\zeta^{N}\Psi^{N})})+2\mu\epsilon(f^{N}):\overline{\epsilon(\zeta^{N}\Psi^{N})}\right]dy\\
 & = & \int_{D_{N}}\left[\lambda\mbox{div}\,f^{N}(\overline{\mbox{div}\,(\Psi^{N})})+2\mu\epsilon(f^{N}):\overline{\epsilon(\Psi^{N})}\right]dy\\
 &  & +\int_{\Omega_{N}'}\left[\lambda\mbox{div}\,f^{N}(\overline{\mbox{div}\,(\zeta^{N}\Psi^{N})})+2\mu\epsilon(f^{N}):\overline{\epsilon(\zeta^{N}\Psi^{N})}\right]dy\\
 & := & II_{1}+II_{2},
\end{eqnarray*}
where 
\begin{eqnarray*}
II_{1} & = & \int_{D_{N}}\left[\lambda\mbox{div}\,f^{N}(\overline{\mbox{div}\,(\Psi^{N})})+2\mu\epsilon(f^{N}):\overline{\epsilon(\Psi^{N})}\right]dy,\\
II_{2} & = & \int_{\Omega_{N}'}\left[\lambda\mbox{div}\,f^{N}(\overline{\mbox{div}\,(\zeta^{N}\Psi^{N})})+2\mu\epsilon(f^{N}):\overline{\epsilon(\zeta^{N}\Psi^{N})}\right]dy.
\end{eqnarray*}
From (\ref{eq:Approximate Solution}) (\ref{eq:Decay Estimate for approx. Sol})
and (\ref{eq:Psi^N}), it is not hard to see 
\[
|II_{2}|=O(e^{-\frac{\sqrt{N}}{2}})\mbox{ as }N\to\infty.
\]
Hence it remains to show that 
\[
\lim_{N\to\infty}N^{m}II_{1}=0
\]
with 
\[
II_{1}=II_{3}+II_{4},
\]
where 
\begin{eqnarray*}
II_{3} & = & \int_{D_{N}}\left[\lambda\mbox{div}\,f^{N}(\overline{\mbox{div}\,(\Phi^{N})})+2\mu\epsilon(f^{N}):\overline{\epsilon(\Phi^{N})}\right]dy,\\
II_{4} & = & \int_{D_{N}}\left[\lambda\mbox{div}\,f^{N}(\overline{\mbox{div}\,(\Psi^{N}-\Phi^{N})})+2\mu\epsilon(f^{N}):\overline{\epsilon(\Psi^{N}-\Phi^{N})}\right].
\end{eqnarray*}
Note that $f^{N}\in H_{0}^{1}(\Omega;\mathbb{R}^{3})$ satisfies $f^{N}=u^{N}-\Phi^{N}$
and 
\begin{equation}
\mathcal{L}f^{N}=-\mathcal{L}\Phi^{N}\mbox{ in }\Omega.\label{eq:f^N equation}
\end{equation}
By using the standard elliptic regularity theory, we have 
\[
\|f^{N}\|_{H_{0}^{1}(\Omega)}\leq C\|\mathcal{L}\Phi^{N}\|_{H^{-1}(\Omega)}\leq C\left\Vert (\sum_{j,k,l=1}^{3}\dot{C}_{ijkl}\dfrac{\partial}{\partial x_{l}}\Phi_{k}^{N})_{i=1}^{3}\right\Vert _{L^{2}(\Omega)}
\]
for some constant $C>0$. By straightforward computation and (\ref{eq:Approximate Solution}),
we have the following lemma.
\begin{lem}
Let $k,l=1,2,3$. For each $b\in\mathbb{N}\cup\{0\}$, there exists
a constant $C_{b}>0$ such that 
\begin{equation}
\left\Vert y_{3}^{b}\dfrac{\partial}{\partial y_{l}}\Phi_{k}^{N}\right\Vert _{L^{2}(\Omega_{N})}\leq C_{b}N^{-\frac{1}{2}+\rho-b}.\label{eq:L^2 for Phi^N}
\end{equation}

\end{lem}
By taking $b=0$, (\ref{eq:L^2 for Phi^N}) will imply that 
\[
\|f^{N}\|_{H_{0}^{1}(\Omega)}\leq CN^{-\frac{1}{2}+\rho}.
\]
Now, $\partial D_{N}=\Gamma_{1}\cup\Gamma_{2}$, where 
\begin{eqnarray*}
\Gamma_{1} & := & \left\{ y:|y_{1}|=N^{\rho-1}\mbox{ or }|y_{2}|=N^{\rho-1},\mbox{ }0\leq y_{3}\leq\dfrac{1}{2\sqrt{N}}\right\} ,\\
\Gamma_{2} & := & \left\{ y:|y_{1}|,|y_{2}|\leq N^{\rho-1},\mbox{ }y_{3}=\dfrac{1}{2\sqrt{N}}\right\} .
\end{eqnarray*}
For $k,l=1,2,3$, it is easy to see that 
\[
\dfrac{\partial}{\partial y_{l}}\Phi_{k}^{N}(y)=0\mbox{ on }\Gamma_{1}\mbox{ and }\dfrac{\partial}{\partial y_{l}}\Phi_{k}^{N}(y)=O(e^{-\frac{1}{2}N^{\frac{1}{2}}})\mbox{ on }\Gamma_{2}\mbox{ as }N\to\infty.
\]
Integration by parts yields that 
\[
II_{3}=-\sum_{i=1}^{3}\int_{D_{N}}f_{i}^{N}(\mathcal{L}\Phi^{N})_{i}dy+O(e^{-\frac{1}{2}N^{\frac{1}{2}}})\mbox{ as }N\to\infty.
\]
Thus, for $i=1,2,3$, by using the Hardy's inequality for $f_{i}^{N}\in H_{0}^{1}(\Omega)$
which was also used in \cite{brown2001recovering,kang2002boundary,nakamura2001local,robertson1997boundary},
we obtain 
\begin{eqnarray}
\int_{D_{N}}f_{i}^{N}(\mathcal{L}\Phi^{N})_{i}dy & \leq & \|y_{3}(\mathcal{L}\Phi^{N})_{i}\|_{L^{2}(D_{N})}\|y_{3}^{-1}f_{i}^{N}\|_{L^{2}(D_{N})}\label{eq:Hardy's inequality}\\
 & \leq & C\|y_{3}(\mathcal{L}\Phi^{N})_{i}\|_{L^{2}(D_{N})}\|f_{i}^{N}\|_{H^{1}(D_{N})}\nonumber \\
 & \leq & C\|y_{3}(\mathcal{L}\Phi^{N})_{i}\|_{L^{2}(D_{N})}N^{-\frac{1}{2}+\rho}.\nonumber 
\end{eqnarray}
By (\ref{eq:Decay Estimate for approx. Sol}), we can see that 
\begin{eqnarray}
\|y_{3}(\mathcal{L}\Phi^{N})_{i}\|_{L^{2}(D_{N})} & \leq & CN^{2-m-\rho}\|y_{3}\mathcal{P}(Ny_{3})e^{-Ny_{3}}\|_{L^{2}(D_{N})}\label{eq:weighted estimates}\\
 & \leq & CN^{-m-1}.\nonumber 
\end{eqnarray}
By (\ref{eq:Hardy's inequality}) and (\ref{eq:weighted estimates}),
we get 
\[
II_{3}=O(N^{-m+\rho-3/2})\mbox{ as }N\to\infty,
\]
which implies 
\[
\lim_{N\to\infty}N^{m}II_{3}=0.
\]

Finally, we need to show that 
\begin{equation}
\lim_{N\to\infty}N^{m}II_{4}=0.\label{eq:II_4}
\end{equation}
Notice that for $i=1,2,3$, 
\[
\Phi_{i}^{N}-\Psi_{i}^{N}=0\mbox{ on }\Gamma_{1}\mbox{ and }\Phi_{i}^{N}-\Psi_{i}^{N}=O(e^{-\frac{1}{2}N^{\frac{1}{2}}})\mbox{ as }N\to\infty.
\]
By using the integration by parts and (\ref{eq:f^N equation}), we
have 
\begin{eqnarray*}
II_{4} & = & -\sum_{i=1}^{3}\int_{D_{N}}(\mathcal{L}f^{N})_{i}(\Phi_{i}^{N}-\Psi_{i}^{N})dy+O(e^{-\frac{1}{2}N^{\frac{1}{2}}})\\
 & = & \sum_{i=1}^{3}\int_{D_{N}}(\mathcal{L}\Phi^{N})_{i}(\Phi_{i}^{N}-\Psi_{i}^{N})dy+O(e^{-\frac{1}{2}N^{\frac{1}{2}}}).
\end{eqnarray*}
We can use the same arguments for $II_{3}$ to show (\ref{eq:II_4}),
which finishes the proof of Theorem 1.1, (2).

\section{Non-flat boundary case\label{Section 4}}

In this section we will consider the boundary determination for the
non-flat boundary case. By using the boundary normal coordinates to
flatten $\partial\Omega$, we will show the necessary change we need
for the non-flat boundary case based on the boundary determination
argument we gave for the flat boundary case. Similar argument was
given in \cite[Section 3]{nakamura2017reconstruction} for the isotropic
elasticity system.

Given any boundary point $x_{0}\in\partial\Omega$, for all $x\in\Omega$
near $x_{0}\in\partial\Omega$, let $y=F(x):\mathbb{R}^{3}\to\mathbb{R}^{3}$
be a $C^{1}$-diffeomorphism which induces the boundary normal coordinates
$y=(y',y_{3})$ such that $F(x_{0})=0$ and $\nabla F(x_{0})=I_{3}$
(a $3\times3$ identity matrix). Let us define the Jacobian matrix
$J:=\nabla F=\left(\dfrac{\partial y_{a}}{\partial x_{r}}\right)_{a,r=1}^{3}$
and denote $G=JJ^{T}=(G_{ai})$, where $J^{T}$ is the transpose of
$J$ and $G(x_{0})=I_{3}$. In addition, near $x_{0}\in\partial\Omega$,

\[
g_{ai}(x)=\sum_{r=1}^{3}\dfrac{\partial x^{a}}{\partial x_{r}}(x)\dfrac{\partial x^{i}}{\partial x_{r}}(x)
\]
satisfying 
\[
g_{33}=1,\mbox{ }g_{a3}=g_{3a}=0\mbox{ for }a=1,2.
\]
Now, we have the following push-forward relations of the elastic tensor
$\dot{\mathcal{C}}$ by 
\begin{equation}
\widetilde{\mathcal{C}}:=F_{*}\dot{\mathcal{C}}=J\dot{\mathcal{C}}J^{T}|_{x=F^{-1}(y)},\label{eq:transformed elastic tensor}
\end{equation}
or componentwisely, $\widetilde{\mathcal{C}}=(\widetilde{\mathcal{C}}_{iqkp})_{1\leq i,q,k,p\leq3}$
with 
\[
\widetilde{\mathcal{C}}_{iqkp}(y)=\left.\left\{ \sum_{j,l=1}^{3}\dot{C}_{ijkl}(x)\dfrac{\partial y_{p}}{\partial x_{l}}\dfrac{\partial y_{q}}{\partial x_{j}}\right\} \right|_{x=F^{-1}(y)}.
\]

It is easy to check that under such localized boundary normal coordinates,
the isotropic elastic equation (\ref{eq:Elasticity}) will become
\begin{equation}
\begin{cases}
(\widetilde{\mathcal{L}}u)_{i}:=\sum_{q,k,p=1}^{3}\dfrac{\partial}{\partial y_{q}}(\widetilde{\mathcal{C}}_{iqkp}\dfrac{\partial}{\partial y_{p}}\widetilde{u}_{k})=0 & \mbox{ in }\{y_{3}>0\},\mbox{ for }i=1,2,3,\\
\widetilde{u}=\widetilde{f} & \mbox{ on }\{y_{3}=0\},
\end{cases}\label{eq:transformed elasticity system}
\end{equation}
where $\widetilde{u}=(F^{-1})^{*}u:=u\circ F^{-1}$ and $\widetilde{f}=f\circ F^{-1}$.
Similar as in Section 2, we can find an approximate solution $\widetilde{\Phi}^{N}(y)$
of \ref{eq:transformation rule 2} with the localized boundary data
$\widetilde{\Phi}^{N}(y',0)=\phi^{N}(y'):=\eta^{N}(y')\mathbf{a}$,
where $\eta^{N}(y')\mathbf{a}\in\mathbb{C}^{3}$ was given by (\ref{eq:boundary test function})
with arbitrary $\mathbf{a}\in\mathbb{C}^{3}$.

As in \cite[Section 3]{nakamura2017reconstruction}, by denoting $\zeta:=(\zeta',\zeta_{3})$,
we can define 
\begin{align*}
 & \widetilde{T}(y,\zeta'):=\left(\widetilde{C}_{i3k3}(y)\right)_{1\leq i,k\leq3},\\
 & \widetilde{R}(y,\zeta'):=\left(\sum_{p=1}^{2}\widetilde{C}_{ipk3}(y)\zeta_{j}\right)_{1\leq i,k\leq3},\\
 & \widetilde{Q}(y,\zeta'):=\left(\sum_{p,q=1}^{2}\widetilde{C}_{ipkq}(y)\zeta_{j}\zeta_{l}\right)_{1\leq i,k\leq3}.
\end{align*}
Likewise Section 2, we need to find a solution of the second order
ordinary differential system with constant matrix variables 
\begin{equation}
\begin{cases}
\widetilde{T}(0)D_{3}^{2}U_{0}+\left(\widetilde{R}(0)+\widetilde{R}^{T}(0)\right)D_{3}U_{0}+\widetilde{Q}(0)U_{0}=0,\\
V^{0}|_{x^{3}=0}=\phi^{N}(y').
\end{cases}\label{eq:tranformed 0-th ODE}
\end{equation}
For that repeat the argument given in Section 2 and need to consider
the following eigenvalue problem 
\begin{equation}
\det\left[\widetilde{T}(0)\Sigma^{2}+\left(\widetilde{R}(0)+\widetilde{R}^{T}(0)\right)\Sigma+\widetilde{Q}(0)\right]=0,\label{eq:eigenvalue problem 2}
\end{equation}
which is similar to \ref{eq:DET2}.

By the transformation rule of tensor, we have 
\begin{equation}
(\sum_{p,q=1}^{3}\widetilde{C}_{iqkp}\zeta_{q}\zeta_{p})_{i,k=1}^{3}=J(\sum_{j,l=1}^{3}\dot{C}_{ijkl}\xi_{j}\xi_{l})_{i,k=1}^{3}J^{T}\label{eq:tensor transformed rule}
\end{equation}
for any $x$ near $x_{0}\in\partial\Omega$ (or for any $y$ near
$0\in\partial\widetilde{\Omega}$, where $\widetilde{\Omega}=F(\Omega)$).
In addition, for any $x\in\partial\Omega$ near $x_{0}$, we can choose
a unit vector $\nu(x)=(\nu_{1},\nu_{2,}\nu_{3})\in\mathbb{R}^{3}$
such that for any $\xi=(\xi_{1},\xi_{2},\xi_{3})\in\mathbb{R}^{3}$
can be represented as $\xi(x)=q\nu(x)+\omega(x,\xi)$ for some $q\in\mathbb{R}$
and $\nu\perp\omega$ and we define 
\begin{align*}
 & \dot{T}:=\left(\sum_{j,l=1}^{3}\dot{C}_{ijkl}\nu_{j}\nu_{l}\right)_{1\leq i,k\leq3},\\
 & \dot{R}:=\left(\sum_{j,l=1}^{3}\dot{C}_{ijkl}\nu_{j}\omega_{l}\right)_{1\leq i,k\leq3},\\
 & \dot{Q}:=\left(\sum_{j,l=1}^{3}\dot{C}_{ijkl}\omega_{j}\omega_{l}\right)_{1\leq i,k\leq3}.
\end{align*}
By \ref{eq:tensor transformed rule}, we also have the following relations
\begin{equation}
\widetilde{T}=J\dot{T}J^{T},\mbox{ }\widetilde{R}=J\dot{R}J^{T}\mbox{ and }\widetilde{Q}=J\dot{Q}J^{T}\label{eq:transformation rule 2}
\end{equation}
in a small neighborhood of $x_{0}\in\partial\Omega$. For solving
the eigenvalue problem \ref{eq:eigenvalue problem 2}, use the relation
\ref{eq:transformation rule 2} and $J$ is an invertible Jacobian
matrix, then it is equivalent to solve 
\begin{equation}
\det\left[\dot{T}\Sigma^{2}+\left(\dot{R}+\dot{R}^{T}\right)\Sigma+\dot{Q}\right]_{x=x_{0}}=0.\label{eq:eigenvalue problem 3}
\end{equation}
Since $\dot{C}_{ijkl}$ is isotropic, $\dot{T}$, $\dot{R}$ and $\dot{Q}$
will not change the forms if we rotate the Cartesian coordinates associated
to this $\xi(x)$, therefore, for any fixed $x$, we may assume $\xi(x)=(\xi',\xi_{3})(x)$,
$\nu(x)=(0,0,1)$ and $\omega(x,\xi)=(\xi',0)$.

In addition, we can construct an approximate solution in terms of
this Cartesian coordinates as we did in Section 2. Hence, we can give
an explicit reconstruction formulae for the Lamé moduli $\lambda(x)$,
$\mu(x)$ and their derivatives from the localized DN map at any $x_{0}\in\partial\Omega$
with $C^{m+2}$-smooth boundary. To be more precise, we will give
the reconstruction formulae to identify the Lamé moduli and their
first order derivatives at the boundary for the non-flat boundary
case in which the effect coming from the transformation of coordinates
and normal vector can be seen very clearly. The reconstruction formulae
are given as follows: For any $\mathbf{a}=(a_{1},a_{2},a_{3})\in\mathbb{C}^{3}$,
let $y=F(x)$ be the map given above, then for any Dirichlet boundary
data $\phi^{N}=\phi^{N}(F(x)\big|_{\partial\Omega})$, where $\phi^{N}(y')=\eta^{N}(y')\exp(\sqrt{-1}Ny'\cdot\omega')\mathbf{a}$,
we have the following approximate solution 
\[
\Phi^{N}(y)=e^{\sqrt{-1}Ny'\cdot\omega'}e^{-Ny_{3}}\left\{ \eta^{N}(y')\mathbf{a}+\sum_{n=1}^{\frac{m}{\rho}}N^{-n\rho}v_{n}(N^{\alpha}y)\right\} ,
\]
with 
\[
\Phi^{N}(F(x))|_{\partial\Omega}=\phi^{N}
\]

1. When $\partial\Omega\in C^{1}$ and $\widetilde{C}_{ijkl}$ is
continuous at $x_{0}\in\partial\Omega$, we have 
\begin{equation}
\lim_{N\to\infty}\left\langle \Lambda_{\widetilde{\mathcal{C}}}\phi^{N},\overline{\phi^{N}}\right\rangle =\sum_{i,j=1}^{3}Z_{ij}(x_{0})a_{i}\overline{a_{j}},\label{eq: zero oder approximation formula}
\end{equation}
where $(Z_{ij})$ is the rank 2 tensor appeared in Theorem \ref{Main Thm}.

2. When $\partial\Omega\in C^{3}$ and $\dot{C}_{ijkl}\in C^{1,p}$
near $x_{0}\in\partial\Omega$, we have 
\begin{align}
 & \lim_{N\to\infty}N\left\langle (\Lambda_{\dot{\mathcal{C}}}-\Lambda_{\dot{\mathcal{C}}^{1}})\phi^{N},\overline{\phi^{N}}\right\rangle \nonumber \\
= & \dfrac{1}{2}\sum_{i,q,k,p=1}^{3}\dfrac{\partial}{\partial y_{3}}\left.\left(\sum_{j,l=1}^{3}\dot{C}_{ijkl}(x)\dfrac{\partial y_{p}}{\partial x_{l}}\dfrac{\partial y_{q}}{\partial x_{j}}\right)\right|_{x=F^{-1}(0)}A_{kp}A_{iq}\nonumber \\
= & \dfrac{1}{4}\dfrac{\partial\lambda}{\partial y_{3}}(x_{0})\left(\sqrt{-1}\sum_{i=1}^{2}\omega_{i}a_{i}-a_{3}\right)^{2}\nonumber \\
 & +\dfrac{1}{2}\dfrac{\partial\mu}{\partial y_{3}}(x_{0})\left[\sum_{i,j=1}^{2}\left(\dfrac{a_{i}\omega_{j}+a_{j}\omega_{i}}{2}\right)^{2}+2\sum_{i=1}^{2}\left(\dfrac{\sqrt{-1}a_{3}\omega_{i}-a_{i}}{2}\right)^{2}+1\right]\nonumber \\
 & +\dfrac{1}{2}\sum_{\underset{0\leq\alpha,\beta,\gamma\leq1}{\alpha+\beta+\gamma=1}}\sum_{i,j,k,l,p,q=1}^{3}\left.\left(\left(\dfrac{\partial^{\alpha}}{\partial y_{3}^{\alpha}}\dot{C}_{ijkl}\right)\dfrac{\partial^{\beta}}{\partial y_{3}^{\beta}}\left(\dfrac{\partial y_{p}}{\partial x_{l}}\right)\dfrac{\partial^{\gamma}}{\partial y_{3}^{\gamma}}\left(\dfrac{\partial y_{q}}{\partial x_{j}}\right)\right)\right|_{x=x_{0}}A_{kp}A_{iq},\label{eq:first order approximate formula}
\end{align}
where the elastic tensor $\dot{\mathcal{C}}^{1}$ is given by 
\[
\dot{\mathcal{C}}^{1}=F^{*}(\widetilde{\mathcal{C}}^{1,0})\mbox{ with }\widetilde{\mathcal{C}}^{1,0}=\widetilde{\mathcal{C}}(y',0),
\]
for $y=(y',y_{3})$ near $0\in\partial F(\Omega)$ and $A_{kp}$ is
a constant rank 2 tensor defined by 
\begin{equation}
A_{kp}=\begin{cases}
\sqrt{-1}\omega_{p}a_{k}, & \mbox{ for }k=1,2,3\mbox{ and }p=1,2,\\
-\omega_{3}a_{k}, & \mbox{ for }k=1,2,3\mbox{ and }p=3.
\end{cases}\label{eq:A_kp}
\end{equation}
We remark here that the boundary determination formulae for the Lame
moduli and their normal derivatives are given in terms of the leading
part of the equations of system. This is really an advantage of scaling
\eqref{eq:scaling} we introduced before.

Since \eqref{eq: zero oder approximation formula} easily follows
by taking into account on the arguments given before the previous
paragraph of this section and $J(x_{0})=I$, we will focus on \eqref{eq:first order approximate formula}.
This formula can be derived by using the integration by parts and
the representation formula of (\ref{eq:transformed elastic tensor}).
By using the same argument given in Section 3, we know that the limit
with respect to $N$ as $N\rightarrow\infty$ of the difference of
DN maps only depends on the highest order term with respect to $N$,
which means we have the following relation 
\begin{equation}
\lim_{N\to\infty}N\left\langle (\Lambda_{\widetilde{\mathcal{C}}}-\Lambda_{\widetilde{\mathcal{C}}^{1,0}})\phi^{N},\overline{\phi^{N}}\right\rangle =\lim_{N\to\infty}N\int_{\Omega}(\mathcal{C}-\dot{\mathcal{C}}^{1})\nabla\Phi^{N}:\nabla\Psi^{N}dx,\label{eq:non-flat reconstruction}
\end{equation}
where $\Phi^{N}(F(x))$, $\Psi^{N}(F(x))$ are approximate solutions
of the differential operators $\nabla\cdot(\dot{\mathcal{C}}\nabla)$
and $\nabla\cdot(\dot{\mathcal{C}}^{1}\nabla)$ with the same boundary
data $\Phi^{N}(F(x))|_{\partial\Omega}=\Psi^{N}(F(x))|_{\partial\Omega}=\phi^{N}$,
respectively.

We will further compute the right hand side of (\ref{eq:non-flat reconstruction})
to obtain \eqref{eq:first order approximate formula}. By the change
of variable $y=F(x)$ and the chain rule, we have 
\begin{align}
 & \int_{\Omega}\sum_{i,j,k,l=1}^{3}\dot{C}_{ijkl}(x)\dfrac{\partial\Phi_{k}^{N}}{\partial x_{l}}\dfrac{\partial\Psi_{i}^{N}}{\partial x_{j}}dx\nonumber \\
= & \int_{F(\Omega)}\sum_{i,p,k,p=1}^{3}\sum_{j,l=1}^{3}\left.\left(\dot{C}_{ijkl}(x)\dfrac{\partial y_{p}}{\partial x_{l}}\dfrac{\partial y_{q}}{\partial x_{j}}\right)\right|_{x=F^{-1}(y)}\dfrac{\partial\Phi_{k}^{N}}{\partial y_{p}}\dfrac{\partial\Psi_{i}^{N}}{\partial y_{q}}dy\nonumber \\
= & \int_{F(\Omega)}\sum_{i,p,k,p=1}^{3}\widetilde{C}_{ipkq}(y)\dfrac{\partial\Phi_{k}^{N}}{\partial y_{p}}\dfrac{\partial\Psi_{i}^{N}}{\partial y_{q}}dy,\label{eq:1 change}
\end{align}
and similarly 
\begin{equation}
\int_{\Omega}\sum_{i,j,k,l=1}^{3}\dot{C}_{ijkl}^{1}(x)\dfrac{\partial\Phi_{k}^{N}}{\partial x_{l}}\dfrac{\partial\Psi_{i}^{N}}{\partial x_{j}}dx=\int_{F(\Omega)}\sum_{i,q,k,p=1}^{3}\widetilde{C}_{ipkp}(y',0)\dfrac{\partial\Phi_{k}^{N}}{\partial y_{p}}\dfrac{\partial\Psi_{i}^{N}}{\partial y_{q}}dy,\label{eq:2 change}
\end{equation}
where $\widetilde{x}$ is the point such that $F(\widetilde{x})=(y',0)$.
Then it is easy to see 
\begin{align}
 & \int_{F(\Omega)}(\widetilde{\mathcal{C}}(y)-\widetilde{\mathcal{C}}^{1,0}(y',0))\nabla_{y}\Phi^{N}:\nabla_{y}\Psi^{N}dy\nonumber \\
= & \int_{F(\Omega)}\sum_{i,q,k,p=1}^{3}\widetilde{C}_{ipkp}(y)\dfrac{\partial\Phi_{k}^{N}}{\partial y_{p}}\dfrac{\partial\Psi_{i}^{N}}{\partial y_{q}}dy\nonumber \\
 & -\int_{F(\Omega)}\sum_{i,q,k,p=1}^{3}\widetilde{C}_{ipkp}(y',0)\dfrac{\partial\Phi_{k}^{N}}{\partial y_{p}}\dfrac{\partial\Psi_{i}^{N}}{\partial y_{q}}dy\label{eq:difference of DN in non-flat}
\end{align}
From direct calculation for the approximate solutions, we have 
\begin{equation}
\nabla_{y}\Phi_{j}^{N}=N\left(\begin{array}{c}
\sqrt{-1}\omega'\\
-1
\end{array}\right)e^{\sqrt{-1}Ny'\cdot\omega'}e^{-Ny_{3}}\eta^{N}(y')a_{j}+O(N^{1-\rho})\label{eq:first order appro}
\end{equation}
and 
\begin{equation}
\nabla_{y}\Phi_{j}^{N}=N\left(\begin{array}{c}
\sqrt{-1}\omega'\\
-1
\end{array}\right)e^{\sqrt{-1}Ny'\cdot\omega'}e^{-Ny_{3}}\eta^{N}(y')a_{j}+O(N^{1-\rho}).\label{eq:first order approxi 2}
\end{equation}
Substitute (\ref{eq:2 change}) and (\ref{eq:2 change}) into (\ref{eq:difference of DN in non-flat}),
by using similar arguments as in Section 3, then we have 
\begin{align}
 & \int_{\Omega}(\dot{\mathcal{C}}(x)-\dot{\mathcal{C}}^{1}(x))\nabla_{y}\Phi^{N}:\nabla_{y}\Psi^{N}dx\nonumber \\
= & \int_{F(\Omega)}(\widetilde{\mathcal{C}}(y)-\widetilde{\mathcal{C}}(y',0))\nabla_{y}\Phi^{N}:\nabla_{y}\Psi^{N}dy\nonumber \\
= & \sum_{i,q,k,p=1}^{3}\int_{F(\Omega)}(\widetilde{C}_{iqkp}-\widetilde{C}_{iqkp}(y',0))\dfrac{\partial\Phi_{k}^{N}}{\partial y_{p}}\dfrac{\partial\Psi_{i}^{N}}{\partial y_{q}}dy.\label{eq:int}
\end{align}
Now, we substitute (\ref{eq:first order appro}) and (\ref{eq:first order approxi 2})
into (\ref{eq:int}), and use (\ref{eq:Main Limit}) again, then we
have 
\begin{align}
 & \lim_{N\to\infty}\int_{\Omega}(\dot{\mathcal{C}}(x)-\dot{\mathcal{C}}^{1}(x))\nabla_{y}\Phi^{N}:\nabla_{y}\Psi^{N}dx\nonumber \\
= & \dfrac{1}{2}\sum_{i,q,k,p=1}^{3}\left(\dfrac{\partial}{\partial y_{3}}\widetilde{C}_{iqkp}\right)(0)A_{kp}A_{iq}.\label{eq:explic}
\end{align}
Note that the quantity $A_{kp}$ is obtained from the representation
of the approximate solution and straightforward calculation.

It remains to give the explicit formula for (\ref{eq:explic}) in
terms of the elastic tensor $\dot{\mathcal{C}}$. By the straightforward
calculation for (\ref{eq:explic}), it is not hard to see that 
\begin{align*}
 & \dfrac{1}{2}\sum_{i,q,k,p=1}^{3}\left(\dfrac{\partial}{\partial y_{3}}\widetilde{C}_{iqkp}\right)(0)A_{kp}A_{iq}\\
= & \dfrac{1}{2}\sum_{i,q,k,p=1}^{3}\dfrac{\partial}{\partial y_{3}}\left.\left(\sum_{j,l=1}^{3}\dot{C}_{ijkl}(x)\dfrac{\partial y_{p}}{\partial x_{l}}\dfrac{\partial y_{q}}{\partial x_{j}}\right)\right|_{x=F^{-1}(0)}A_{kp}A_{iq}\\
= & \dfrac{1}{4}\dfrac{\partial\lambda}{\partial y_{3}}(x_{0})\left(\sqrt{-1}\sum_{i=1}^{2}\omega_{i}a_{i}-a_{3}\right)^{2}\\
 & +\dfrac{1}{2}\dfrac{\partial\mu}{\partial y_{3}}(x_{0})\left[\sum_{i,j=1}^{2}\left(\dfrac{a_{i}\omega_{j}+a_{j}\omega_{i}}{2}\right)^{2}+2\sum_{i=1}^{2}\left(\dfrac{\sqrt{-1}a_{3}\omega_{i}-a_{i}}{2}\right)^{2}+1\right]\\
 & +\dfrac{1}{2}\sum_{\underset{0\leq\alpha,\beta,\gamma\leq1}{\alpha+\beta+\gamma=1}}\sum_{i,j,k,l,p,q=1}^{3}\left.\left(\left(\dfrac{\partial^{\alpha}}{\partial y_{3}^{\alpha}}\dot{C}_{ijkl}\right)\dfrac{\partial^{\beta}}{\partial y_{3}^{\beta}}\left(\dfrac{\partial y_{p}}{\partial x_{l}}\right)\dfrac{\partial^{\gamma}}{\partial y_{3}^{\gamma}}\left(\dfrac{\partial y_{q}}{\partial x_{j}}\right)\right)\right|_{x=x_{0}}A_{kp}A_{iq}
\end{align*}
which proves the reconstruction formula to identify the first order
derivatives of the Lamé moduli at the boundary for the non-flat boundary
case.

 \bibliographystyle{plain}
\bibliography{ref}

\end{document}